\documentclass[a4paper,11pt]{amsart}
\usepackage{amssymb,amsfonts,amsxtra,
mathrsfs,placeins,graphicx,verbatim,stmaryrd}
\usepackage[all]{xy}
\xyoption{line}
\usepackage{fullpage}
\usepackage{euscript}
\newtheorem{theorem}{Theorem}[section]
\newtheorem{cor}[theorem]{Corollary}

\newtheorem{prop}[theorem]{Proposition}

\newtheorem{example}[theorem]{Example}
\newtheorem{defi}[theorem]{Definition}
\newtheorem{rem}[theorem]{Remark}

\numberwithin{equation}{section}
\DeclareMathOperator{\Hom}{Hom}
\DeclareMathOperator{\MC}{MC}
\DeclareMathOperator{\Aut}{Aut}

\DeclareMathOperator{\Ker}{Ker}

\newcommand{\noproof}{\begin{flushright} \ensuremath{\square}
\end{flushright}}
\def\falg{{\mathscr F}Alg}

\def\ground{\mathbf k}
\def\C{\mathscr C}
\def\L{\mathscr L}
\def\Gr{\operatorname{Gr}}

\def\g{\mathfrak g}
\def\h{\mathfrak h}
\def\MCmod{\mathscr {MC}}
\def\End{\operatorname{End}}
\def\Z{\mathbb{Z}}

\def\id{\operatorname{id}}

\def\CE{\operatorname{CE}}
\DeclareMathOperator{\Harr}{Harr}
\DeclareMathOperator{\Hoch}{Hoch}
\thanks{}

\begin{document}

\title[Models for function spaces]{Maurer-Cartan moduli and models for function spaces}
\author{A.~Lazarev}
\thanks{The author is grateful to C. Braun, J. Chuang, V. Hinich, B. Keller and M. Markl for many useful discussions concerning the subject of this paper.}
\address{University of Lancaster\\ Department of
Mathematics and Statistics\\Lancaster LA1 4YF, UK.}
\email{a.lazarev@lancaster.ac.uk} 
\begin{abstract}
We set up a formalism of Maurer-Cartan moduli sets for $L_\infty$ algebras and associated twistings based on the closed model category structure on formal differential graded algebras (a.k.a. differential graded coalgebras). Among other things this formalism allows us to give a compact and manifestly homotopy invariant treatment of Chevalley-Eilenberg and Harrison cohomology. We apply the developed technology to construct rational homotopy models for function spaces.
\end{abstract}

\maketitle
\tableofcontents
\section{Introduction}
The homotopy theory of function spaces has been much studied from various standpoints, cf. \cite{smith} for a comprehensive and up to date survey. The basic problem is as follows: given two spaces $X$ and $Y$ which are sufficiently nice (e.g. CW-complexes of finite type) describe effectively the homotopy type of the mapping space $F(X,Y)$ in terms of the homotopy types of $X$ and $Y$. From the point of view of rational homotopy theory several answers are known, \cite{Hae, BS, BPS}, but these answers are complicated and not readily amenable to calculations. We also mention the recent papers \cite{BFM, BFM'} and the preprint \cite{Ber} where explicit Lie models for function spaces were given but again, these models were not formulated in the standard framework of derived functors of homological algebra.  In \cite{BL} the homotopy groups of function spaces were computed in terms of Harrison-Andr\'e-Quillen cohomology, however the methods of that paper do not extend to a construction of full-fledged rational homotopy models of function spaces.

In the present paper we fill this gap. Given rational models (Sullivan or Lie-Quillen) of $X$ and $Y$ we construct a rational model for $F(X,Y)$ (as well as for the based function space $F_*(X,Y)$) in terms of traditional derived functors (Harrison and Chevalley-Eilenberg complexes). The construction is based on the formalism of Maurer-Cartan moduli sets and Maurer-Cartan twistings; some parts of this formalism are certainly known to experts but are difficult to locate in the literature; we hope that a unified treatment presented here will be of independent interest, especially from the standpoint of algebraic deformation theory.

The paper is organized as follows. In Section 2 we recall the construction of a closed model category structure on the category of formal commutative differential graded algebras due to Hinich; note that Hinich (as well as some other authors) prefers to work with coalgebras; we feel that the equivalent language of formal algebras is more natural, particularly in connection with Maurer-Cartan sets. Section 3 contains a description of cofibrant objects in Hinich's closed model category; these turn out to be formal cdgas representing $L_\infty$ algebras. Section 4 discusses the Neisendorfer closed model category structure which could be viewed as a localization of Hinich's (the latter point of view is not pursued here). This is needed for the applications to rational homotopy theory that we have in mind. In Section 5 we introduce the notion of a Maurer-Cartan set and prove a general statement about its homotopy invariance. This result underlies the modern approach to deformation theory suggested by Deligne, Feigin and Drinfeld at the end of the 1980's; it is also related to the still older manuscript of Schlessinger and Stasheff, which has recently been published on the arXiv, \cite{SS}. Various versions of it have since been proved by many authors. Our version is probably both the most general and simplest to prove (but it relies on the deep results of Hinich on the existence of a closed model category structure on formal cdgas).

Section 6 is devoted to the notion of a Maurer-Cartan twisting in the setting of $L_\infty$ algebras and Section 7 describes how Harrison and Chevalley-Eilenberg cohomology can be described compactly in terms of twistings. Finally, Section 6 contains the main application of the developed apparatus: the construction of a Lie-Quillen model of a function space between two rational nilpotent spaces.

\subsection{Notation and conventions} We work in the category of $\Z$-graded vector spaces over a field $\ground$
of characteristic zero, and we avoid mentioning $\ground$ explicitly. When considering models for topological spaces the field $\ground$ is understood to be $\mathbb Q$. Differential graded algebras will have cohomological grading with upper indices and differential graded Lie algebras will have homological grading with lower indices, unless indicated otherwise. The degree of a homogeneous element $x$ in a graded vector space is denoted by $|x|$.  The \emph{suspension}  $\Sigma V$ of a homologically graded vector space $V$ is defined by the convention $\Sigma V_i=V_{i-1}$; for a cohomologically graded space the convention is as follows: $\Sigma V^i= V^{i+1}$. The functor of taking the linear dual takes homologically graded vector spaces into cohomologically graded ones so that $(V^*)^i=(V_{i})^*$; further we will write $\Sigma V^*$ for $\Sigma(V^*)$; with this convention there is an isomorphism $(\Sigma V)^*\cong\Sigma^{-1}V^*$.

The adjective `differential graded will be abbreviated as `dg'. A
(commutative) differential graded (Lie) algebra will be abbreviated
as (c)dg(l)a.  We will often invoke the notion of a \emph{formal}
(dg) vector space; this is just an inverse limit of
finite-dimensional vector spaces. Here the notion of formality is understood in the sense of a `formal neighborhood' rather than `being quasi-isomorphic to the cohomology'; in a couple of places where this clashes with the standard terminology of rational homotopy theory the distinction is specifically spelled out. An example of a formal space is
$V^*$, the $\ground$-linear dual to a discrete vector space~$V$. A
formal vector space comes equipped with a topology and whenever we
deal with a formal vector space all linear maps from or into it will
be assumed to be continuous; thus we will always have $V^{**}\cong
V$.  All of our unmarked tensors are understood to be taken over
$\ground$. The tensor product $V\otimes W$ of two formal spaces is understood to be the completed tensor product (and so, it will again be formal). If $V$ is a discrete space and $W=\lim_\leftarrow{W_i}$ is
a formal space we will write $V\otimes W$ for
$\lim_{\leftarrow}V\otimes W_i$; thus for two discrete spaces $V$ and
$U$ we have $\Hom(V,U)\cong U\otimes V^*$.

For two topological spaces $X$ and $Y$ we will write $[X,Y]$ for the set of homotopy classes of maps $X\to Y$; if $X$ and $Y$ are pointed spaces
then $[X,Y]_*$ will denote the set of pointed homotopy classes of such maps.
\section{Hinich's closed model category structure}
Consider the category of dg cocommutative coassociative coalgebras which are \emph{cocomplete} in the sense that the filtration given by the kernels of the iterated comultiplication is exhaustive. For a homologically graded coalgebra $X$ there is defined a dgla ${\mathscr L}(X)$ which is the free Lie algebra on $\Sigma X$ and the differential is induced in the standard way by the comultiplication on $X$ and by the internal differential on $X$. Then Hinich proved in \cite{H} that this category could be turned into a closed model category where the weak equivalences are those maps $X\to Y$ of dg coalgebras for which ${\mathscr L}(X)\to{\mathscr L}(Y)$ are quasi-isomorphisms of dglas.

Note that the linear dual to a cocomplete coalgebra $X$ is a (non-unital) algebra $X^*$ which is \emph{formal} in the sense that \[X^*=\lim_\leftarrow X_n^*;\]
here $X^*_n$ is the cokernel of the n-fold multiplication map $X^{*\otimes n}\to X^*$. The filtration by the kernels of the maps $X^*\to X^*_n$ will be called the \emph{canonical} filtration on $X^*$; it is complete and Hausdorff. Some of the elementary properties of formal cdgas are discussed in the appendix to \cite{HL}.

It follows that there is a closed model category structure on formal cdgas which we will explicitly describe; this category will be denoted by $\falg$. Thus, an object in $\falg$ is a formal non-unital cdga and morphisms are required to be continuous with respect to the profinite topology. Note that sometimes it is  more convenient
to consider the category of \emph{unital} and \emph{augmented} formal cdgas; the morphisms will then be required to respect the unit and the augmentation. The latter category is clearly equivalent to the category of $\falg$. Indeed  a non-unital formal cdga $A$ determines an augmented unital one: $\tilde{A}=A\oplus\ground$, obtained from $A$ by adjoining a unit and conversely, the augmentation ideal $B_+$ of an augmented formal cdga $B$ is a non-unital formal cdga. We will use the term `formal cdga' to mean `non-unital formal cdga' unless stated otherwise. Note that the canonical filtration on $A$ corresponds to the filtration on $\tilde{A}$ by the powers of the maximal ideal in $\tilde{A}$.

Any formal cdga determines a dgla as follows.
\begin{defi}\label{har}
Let $A$ be a formal cdga and set $\L(A)$ to be a dgla whose underlying space is the free Lie algebra on $\Sigma A^*$ and the differential $d$ is defined as $d=d_{I}+d_{II}$; here $d_I$ is induced by the internal differential on $A$ and $d_{II}$ is determined by its restriction onto $\Sigma A^*$ which is in turn induced by the product map $A\otimes A\to A$.
\end{defi}
\begin{rem}
Note that since $A$ is formal its dual $A^*$ is discrete and thus, the dgla $\L(A)$ is a conventional dgla (with no topology). The construction $\L(A)$ is the continuous version of the Harrison complex associated with a cdga.
\end{rem}
\begin{defi}
A morphism $f:A\to B$ in $\falg$ is called
\begin{enumerate}
\item
a \emph{weak equivalence} if
${\mathscr L}(f):{\mathscr L}(B)\to{\mathscr L}(A)$
is a quasi-isomorphism of dglas;
\item
a \emph{fibration} if $f$ is surjective; if, in addition, $f$ is a weak equivalence then $f$ is called an \emph{acyclic fibration};
\item
a \emph{cofibration} if $f$ has the left lifting property with respect to all acyclic fibrations. That means that in any commutative square
\[
\xymatrix{A\ar_f[d]\ar[r]&C\ar^g[d]\\
A\ar@{-->}[ur]\ar[r]&D}
\]
where $g$ is an acyclic fibration there exists a dotted arrow making the whole diagram commutative.
\end{enumerate}
\end{defi}
\begin{theorem}
The category $\falg$ is a closed model category with fibrations, cofibrations and weak equivalences defined as above.
\end{theorem}
\begin{proof}
One has to observe only that the category $\falg$ is anti-equivalent to the category of cocomplete coalgebras of Hinich: given any cocomplete dg coalgebra $X$, its linear dual $X^*$ is a formal cdga, and the continuous dual to a formal cdga is a cocomplete dg coalgebra.
\end{proof}
The functor $\L$ from $\falg$ to the category $\L ie$ of dglas admits an adjoint functor $\C:\L ie\to\falg$ defined as follows.
For a dgla $\g$ set $\C(\g)=\hat{S}\Sigma^{-1} \g^*$, the completed symmetric algebra on $\g^*$.
The differential $d$ on $\C(\g)$ is defined as
$d=d_{I}+d_{II}$; here $d_I$ is induced by the internal differential on $\g$ and
$d_{II}$ is determined by its restriction onto $\Sigma^{-1} \g^*$ which is in turn induced by the bracket map $\g\otimes \g\to \g$.
\begin{rem}
The formal cdga $\C(\g)$ is otherwise known as the Chevalley-Eilenberg cohomology complex of the dgla $\g$.
\end{rem}
\begin{rem}
We will occasionally need the following mild generalization of the functors $\C$ and $\L$. Let $A$ be a cdga. For a dgla $\g$ we will call $A\otimes \g$ an \emph{ $A$-linear dgla}; similarly if $B$ is a formal cdga then  $A\otimes B$ is an \emph{$A$-linear formal cdga}. The morphisms of $A$-linear dgla and formal cdgas are required to be morphisms of dg $A$-modules. We define:
\[\C_A(A\otimes \g):=\hat{S}_A\Sigma^{-1}(A\otimes \g^*)\cong A\otimes\hat{S}\Sigma^{-1} \g^*\cong A\otimes\C(\g);\]
the differential in $\C_A(A\otimes \g)$ is the tensor product of the differential in $A$ and in $\C(\g)$. Thus, $\C_A(A\otimes \g)$ is an $A$-linear formal cdga. Likewise $\L_A(A\otimes B)$ is the $A$-linear dgla $A\otimes \L(B)$. Clearly the functors $\C_A$ and $\L_A$ are adjoint functors between the categories of $A$-linear dglas and $A$-linear dglas.
\end{rem}

The category $\L ie$ is itself a closed model category where weak equivalences are quasi-isomorphisms of dglas and fibrations are surjective maps.

The following result is proved in \cite{H}.
\begin{theorem}\label{Hinich}\
\begin{itemize}
\item
For any formal cdga $A$ and a dgla $\g$ here is a natural isomorphism of sets
\[
\Hom_{\falg}(\L(A),\g)\cong\Hom_{L ie}(\C(\g),A).
\]
\item
The functor $\L$ converts fibrations and acyclic fibrations in $\falg$ into cofibrations and acyclic cofibrations in $\L ie$ respectively.
\item
The functor $\C$ converts fibrations and acyclic fibrations in $\L ie$ into cofibrations and acyclic cofibrations in $\falg$ respectively.
\item
The adjunction maps $p=p(\g):\L\C(\g)\to\g$ and $q=q(A):\C\L(A)\to A$ are weak equivalences of dglas and formal cdgas respectively; thus
the functors $\C$ and $\L$ induce inverse equivalences of the homotopy categories of $\falg$ and $\L ie$.
\end{itemize}
\end{theorem}
An important special case of a weak equivalence in $\falg$ is a \emph{filtered quasi-isomorphism}.
\begin{defi}\
\begin{itemize}
\item
An \emph{admissible filtration} on a formal cdga $A$ is a filtration of the form $A=F_1(A)\supset F_2(A)\supset\ldots$ which is Hausdorff, i.e. $\bigcap_{p>0}F_p(A)=0$. In addition, an admissible filtration is required to be \emph{multiplicative}, i.e. $F_p(A)\cdot F_q(A)\subset F_{p+q}(A)$. The associated graded cdga is defined as $\Gr(A)=\bigoplus_{p>0}F_p(A)/F_{p+1}(A)$.
\item
A map $f:A\to B$ in $\falg$ is a filtered quasi-isomorphism if $f$ induces a quasi-isomorphism on the associated graded cdgas for some admissible filtrations on $A$ and $B$.
\end{itemize}
\end{defi}
We will refer to a formal cdga endowed with an admissible filtration as simply `filtered formal cdga'. The canonical filtration on a formal cdga is an example of an admissible filtration. Note that an admissible filtration $F_p$ on a formal cdga $A$ is always \emph{complete}, i.e. ${\underrightarrow{\lim}}A/F_p(A)=A$ owing to the vanishing of ${\underrightarrow{\lim}}^1$ on the category of formal vector spaces.

We have the following result proved in the language of coalgebras in \cite{H}, Proposition 4.4.4.
\begin{prop}\label{filter}
A filtered quasi-isomorphism between filtered formal cdgas is necessarily a weak equivalence.
\end{prop}
\noproof
\section{$L_\infty$ algebras}
The notion of a weak equivalence in $\falg$ is somewhat obscure and one wants to have a more explicit characterization of them.  In this section we will give such a characterization for \emph{cofibrant} cdgas and also show that the latter are, effectively, the same as  $L_\infty$ algebras.
\begin{defi}\
\begin{itemize}
\item
Let $(V,d)$ be a (homologically) graded dg vector space; then an $L_\infty$ algebra structure on $V$ is a continuous derivation $m$ of the completed symmetric algebra $\hat{S}\Sigma^{-1} V^*$ such that $m$ has no constant and linear terms and $(m+d)^2=0$. We will call $\hat{S}\Sigma^{-1} V^*$ supplied with the differential $m+d$ the \emph{representing formal cdga} of $V$. The homogeneous components of $m$ will be denoted by $m_i$, so that $m_i:\hat{S}^i\Sigma^{-1}V^*\to \Sigma^{-1}V^*$.
If $V$ itself has vanishing differential, then $(V,m)$ is called a \emph{minimal} $L_\infty$ algebra.
\item
For two $L_\infty$ algebras $(V,m_V)$ and $(U,m_U)$ an $L_\infty$ \emph{morphism} $U\to V$ is a continuous map of formal cdgas $\hat{S}\Sigma^{-1} V^*\to\hat{S}\Sigma^{-1} U^*$. It is called an $L_\infty$ quasi-isomorphism if it induces a quasi-isomorphism of dg vector spaces $\Sigma^{-1} U^*\to\Sigma^{-1} V^*$ (or equivalently of dg vector spaces $U\to V$).
\end{itemize}
\end{defi}
One of the central results about $L_\infty$ algebras is the \emph{decomposition theorem} for $L_\infty$ algebras (Lemma 4.9 of \cite{Kon}) which was later generalized and reproved by several authors. To formulate it, let us introduce an $L_\infty$ algebra represented by the formal cdga $L(x,y)=\hat{S}(x,y)$ with $|y|=|x|+1$ and $m(x)=y$. This is an \emph{elementary linear contractible} $L_\infty$ algebra. A general linear contractible $L_\infty$ algebra is a direct product of elementary ones (this corresponds to the tensor product of representing formal cdgas). Then one has the following result whose proof could be found in the mentioned paper of Kontsevich; a non-inductive proof, also valid in the operadic generality is contained in \cite{CL'}.
\begin{theorem}\label{decomp}
Any $L_\infty$ algebra $(V,m)$ is $L_\infty$ isomorphic to a direct product of a linear contractible $L_\infty$ algebra and a minimal one. The latter is determined uniquely up to a non-canonical isomorphism.
\end{theorem}
\noproof
\begin{prop}\label{char}\
\begin{enumerate}
\item
The map $\hat{S}\Sigma^{-1} V^*\to\hat{S}\Sigma^{-1} U^*$ representing an $L_\infty$ quasi-isomorphism $(U,m_U)\to(V,m_V)$ is a weak equivalence in $\falg$; conversely any weak equivalence $\hat{S}\Sigma^{-1} V^*\to\hat{S}\Sigma^{-1} U^*$ represents an $L_\infty$ quasi-isomorphism of the corresponding $L_\infty$ algebras.
\item
Any formal cdga $\hat{S}\Sigma^{-1} V^*$ determining an $L_\infty$ structure on a space $V$ is a cofibrant object in $\falg$. Conversely, any cofibrant formal cdga is a formal cdga $\hat{S}\Sigma^{-1} V^*$ representing an $L_\infty$ algebra.
\end{enumerate}
\end{prop}
\begin{proof}
Note first, that an $L_\infty$ quasi-isomorphism $\hat{S}\Sigma^{-1} V^*\to\hat{S}\Sigma^{-1} U^*$ is a filtered quasi-isomorphism with respect to the canonical filtrations on $\hat{S}\Sigma^{-1} V^*$ and $\hat{S}\Sigma^{-1} U^*$ and it follows by Proposition \ref{filter}  that it is a weak equivalence. Conversely, let $f:\hat{S}\Sigma^{-1} V^*\to\hat{S}\Sigma^{-1} U^*$ be a weak equivalence between
two formal cdgas; such a map is clearly an $L_\infty$ morphism and we need to show that $f$ induces a quasi-isomorphism on the spaces of indecomposables: $\Sigma^{-1} V^*\to\Sigma^{-1} U^*$. Since $f$ is a weak equivalence the dgla map
\[
\L(f):\L(\hat{S}\Sigma^{-1} V^*)\to \L(\hat{S}\Sigma^{-1} U^*)
\]
is a quasi-isomorphism.

Let us now compute the homology of $\L(\hat{S}\Sigma^{-1} V^*)$ and $\L(\hat{S}\Sigma^{-1} U^*)$. The filtration by the powers of the augmentation ideal on  $\hat{S}\Sigma^{-1} V^*$ determines an increasing filtration on $\L(\hat{S}\Sigma^{-1} V^*)$. Since the internal differential of the formal cdga $\hat{S}\Sigma^{-1} V^*$ vanishes on its associated graded it follows that the differential $d_1$ in the spectral sequence converging to the homology of $\L(\hat{S}\Sigma^{-1} V^*)$ is the Harrison differential $d_{II}$ (see Definition \ref{har}). Since $\hat{S}\Sigma^{-1} V^*$ is a (completed) free graded commutative algebra the homology with respect to this differential reduces to $V$. Thus, the spectral sequence collapses and
\[
H^*[\L(\hat{S}\Sigma^{-1} V^*]\cong H_*(V).
\]
Arguing similarly, we obtain
\[
H^*[\L(\hat{S}\Sigma^{-1} U^*]\cong H_*(U)
\]
and since the spectral sequences for $\L(\hat{S}\Sigma^{-1} V^*)$ and $\L(\hat{S}\Sigma^{-1} U^*)$ map into one another under $\L(f)$, we conclude that
$f$ induces a quasi-isomorphism $\Sigma^{-1} V^*\to\Sigma^{-1} U^*$ as required. Part (1) is therefore proved.

Let us prove (2). Consider an $L_\infty$ algebra $(V,m)$ and recall that by the decomposition theorem it is isomorphic to a product of linear contractible $L_\infty$ algebras and a minimal one. Since linear contractible $L_\infty$ algebras are obviously represented by cofibrant formal cdgas we may assume that $(V,m)$ is minimal from the start. Consider its representing formal cdga $A=\hat{S}\Sigma^{-1} V^*$ and the adjunction map
$q:\C\L(A)\to A$. The latter map represents an $L_\infty$ quasi-isomorphism and so it has a quasi-inverse $i:A\to\C\L(A)$. We see that $q\circ i$ represents a quasi-isomorphism of $A$ onto itself which is chain homotopic to the identity when restricted onto the indecomposables of $A$; since $A$ represents a minimal $L_\infty$ algebra $q\circ i$ must be the identity and so $A$ is a retract of $\C\L(A)$. Since the latter is cofibrant so is $A$.

Conversely, for any cofibrant formal cdga $A$ there is a surjective weak equivalence $\C\L(A)\to A$ which will necessarily admit a section (since $A$ is cofibrant) exhibiting $A$ as a retract of a formal cdga representing an $L_\infty$ algebra. Finally, it is clear that a (continuous) retract of a formal power series algebra is a formal power series algebra itself which concludes the proof of (2).
\end{proof}
\begin{rem}\
\begin{enumerate}\label{filt}
\item
We see, therefore, that as long as we are interested in the homotopy theoretical problems in the category $\falg$, we can restrict ourselves to considering only those formal cdgas which represent $L_\infty$ algebras; a weak equivalence between such objects will then be a usual $L_\infty$ quasi-isomorphism. One could go even further and consider only those formal cdgas representing minimal $L_\infty$ algebras; a weak equivalence will in this case reduce to a mere isomorphism.
\item
It is not true, in general, that any weak equivalence in $\falg$ must be a filtered quasi-isomorphism with respect to canonical filtrations (although for cofibrant formal cdgas it is true by Proposition \ref{char}). Here is a simple example. Let $A$ be the formal cdga spanned by a single vector $x$ with $|x|=1$ and with zero multiplication (the augmented version of $A$ is the algebra $\tilde{A}=\ground[x]/x^2$). Then it is easy to see that $\C\L(A)$ is the formal cdga $B=\ground[x,y]_+$ with $|y|=1$ and $d(y)=x^2$. Further, the obvious map $f:B\to A$ sending $y$ to zero is a weak equivalence in $\falg$, but it is not a filtered quasi-isomorphism with respect to the canonical filtration on $B$. The `correct' filtration on $B$ is specified by $F_1(B)=B$ and $F_p(B)=(yx^{p-2},x^p), p\geq 2$, the ideal in $\ground[x,y]$ generated by $yx^{p-2}$ and $x^p$ rather than the $p$th power of the maximal ideal. Under this filtration the map $f$ is a filtered quasi-isomorphism.
\item
On the other hand, it is true that if $f:A\to B$ is a weak equivalence in $\falg$ then $A$ and $B$ can be connected by a zig-zag of maps which are filtered weak equivalences (but not necessarily with respect to the canonical filtrations). Indeed, we have the following commutative diagram:
\[
\xymatrix
{
\C\L(A)\ar^{\C\L(f)}[r]\ar[d]&\C\L(B)\ar[d]\\
A\ar^f[r]&B
}
\]
where the vertical arrows are the adjunction morphisms. These adjunction morphisms are filtered quasi-isomorphisms by \cite{H}, Proposition 4.4.3 and the upper horizontal map is a filtered quasi-isomorphism as a weak equivalence between cofibrant formal cdgas.
\end{enumerate}
\end{rem}
Next we will discuss the notion of a \emph{homotopy} in the category $\falg$. To this end, consider the cdga $\ground[z,dz]$ the free graded commutative algebra on the generators $z$ and $dz$ with $|z|=0,|dz|=1$ and $d(z):=dz$; this is the familiar de Rham algebra of forms on an interval. Specializing $z$ to $1$ and $0$ gives cdga maps $|_1,|_0:\ground[z,dz]\to \ground$. Given a formal cdga $A$, we will consider $A[z,dz]$, the tensor product of $A$ and $\ground[z,dz]$. There are two specialization maps $A[z,dz]\to A$ which we will denote $|_1$ and $|_0$ as above.

\begin{defi}\
\begin{enumerate}
\item
Let $A,B$ be formal cdgas and $f,g:A\to B$ be maps in $\falg$.  Then $f$ and $g$ are said to be \emph{Sullivan homotopic} if there exists a continuous cdga map $h: A\to B[z,dz]$ such that $|_0\circ h=f$ and $|_1\circ h=g$.
\item
Two formal cgdas $A$ and $B$ are called \emph{Sullivan homotopy equivalent} if there are maps $f:A\to B$ and $g:B\to A$ such that $f\circ g$ and $g\circ f$ are Sullivan homotopic to $\id_B$ and $\id_A$ respectively.
\end{enumerate}
\end{defi}
Recall that one also have a similar notion of Sullivan homotopy in the category of dglas: two dgla maps $f,g:\g\to \h$
are Sullivan homotopic if there is a dgla map $\g\to\h[z,dz]$
restricting to $f$ and $g$ at $z=0$ and $z=1$ respectively. Furthermore, the dgla $\h[z,dz]$ is a path object for $\h$ and thus, the Sullivan homotopy for dglas is an instance of the closed model theoretic notion  of \emph{right homotopy}. Similarly, one would like to be able to view $B[z,dz]$ as a path object for $B$ and a Sullivan homotopy as a right homotopy in the closed model category of formal cdgas.  Such an interpretation is not available since $\ground[z,dz]$ is \emph{not} a formal cdga. Fortunately, the following result shows that for all practical purposes one can treat $B[z,dz]$ as if it were a path object for $B$.
\begin{theorem}\label{path}
Let $A$ be a cofibrant formal cdga, $B$ be an arbitrary formal cdga and $f,g:A\to B$ be two maps in $\falg$.  Then $f$ and $g$ are Sullivan homotopic if and only if they are homotopic, i.e. represent the same maps in the homotopy category of $\falg$.
Thus, the set of equivalence classes of Sullivan homotopic maps $A\to B$ is bijective with the set of maps from $A$ into $B$ in the homotopy category of $\falg$. This set will be denoted by $[A,B]_*$.
\end{theorem}
\begin{proof}
Let $h:A\to B[z,dz]$ be a homotopy between $f$ and $g$. This homotopy can be viewed as a $\ground[z,dz]$-linear map $A[z,dz]\to B[z,dz]$,
i.e. as a map of $\ground[z,dz]$-linear formal cdgas. Applying the functor $\L_{\ground[z,dz]}$ to it we obtain a $\ground[z,dz]$-linear map
\[\L(B)[z,dz]\cong \L_{\ground[z,dz]}(B[z,dz])\to \L_{\ground[z,dz]}(A[z,dz])\cong\L(A)[z,dz]\] which is the same as a dgla homotopy
\[\L(B)\to\L(A)[z,dz].\] Clearly at $z=0,1$ the above homotopy specializes to the dgla maps $\L(f), \L(g):\L(B)\to \L(A)$ respectively.
It follows that the maps $\L(f)$ and $\L(g)$ are homotopic and thus, $f$ and $g$ are homotopic as well.

Conversely, suppose that $f$ and $g$ are homotopic. This implies that the dgla maps $\L(f)$ and $\L(g): \L(B)\to\L(A)$ are homotopic and thus (since $\L(B)$ is a cofibrant dgla) there exists a Sullivan dgla homotopy
$\L(B)\to\L(A)[z,dz]$
which restricts to $\L(f)$ and $\L(g)$ at $z=0,1$. Viewing the last homotopy as a $\ground[z,dz]$-map and applying the functor $\L_{\ground[z,dz]}$ we obtain a map of formal $\ground[z,dz]$-linear cdgas
$\C\L(A)[z,dz]\to\C\L(B)[z,dz]$
which is the same as a Sullivan homotopy
\[s:\C\L(A)\to\C\L(B)[z,dz]\] which specializes to $\C\L(f)$ and $\C\L(g)$ at $z=0$ and $z=1$ respectively.
Consider the following diagram
\[
\xymatrix
{
\C\L(A)\ar^s[r]\ar_{q(A)}[d]&\C\L(B)[z,dz]\ar^{q(B[z,dz])}[d]\\
A\ar_i@{-->}@/_/[u]&B[z,dz]
}
\]
Here the dotted map $i$ is a section of $q(A)$ (which exists because $A$ is a cofibrant formal cdga). Now the required homotopy $h:A\to B[z,dz]$
is defined by the formula $h=q(B[z,dz])\circ s\circ i$.
\end{proof}
\begin{cor}\label{hin}
Let $A$ be a cofibrant formal cdga and $B$ be an arbitrary formal cdga:
\begin{enumerate}
\item
The relation of Sullivan homotopy on the set of maps $A\to B$ is an equivalence relation.
\item
If $A^\prime$ is a cofibrant formal cdga weakly equivalent to $A$ and $B^\prime$ is weakly equivalent to $B$ then there is a bijective correspondence $[A,B]_*\cong[A^\prime,B^\prime]_*$.
\item
Two cofibrant formal cdgas are Sullivan homotopy equivalent if and only if they are weakly equivalent.
\end{enumerate}
\end{cor}
\begin{proof}
All these statements follow from Theorem \ref{path} in a straightforward fashion. For example, if $f:A\to B$ is a weak equivalence
between two cofibrant formal cdgas then $f$ is invertible in the homotopy category of $\falg$; the inverse map is then represented by a map $g:B\to A$ since $B$ is cofibrant and all formal cdgas are fibrant.
We have $f\circ g$ is homotopic to $\id_B$ and so $f\circ g$ is Sullivan homotopic to
$\id_B$; similarly $g\circ f$ is Sullivan homotopic to $\id_A$. Other claims are equally obvious.
\end{proof}
From now on we will refer to Sullivan homotopy as simply homotopy (assuming that the source is cofibrant as it will always be).
\subsection{Free homotopy} We will now discuss the notion of a free homotopy in the category $\falg$. Unfortunately, this notion does appear slightly ad hoc and from the abstract point of view it seems unjustified. The reason for introducing it is that it corresponds to the notion of a \emph{not necessarily basepoint-preserving} homotopy between connected spaces and this correspondence will be made precise later on in the paper. Recall that for a formal cdga $A$ we denoted by $\tilde{A}$ the algebra obtained from $A$ by adjoining a unit.
\begin{defi}
Let $f,g:A\to B$ be a map between two formal cdgas and $\tilde{f},\tilde{g}:\tilde{A}\to \tilde{B}$ are the corresponding maps between the unital algebras $\tilde{A}$ and $\tilde{B}$. We say that $f$ and $g$ are \emph{freely homotopic} if there exists a (continuous) map
$h:\tilde{A}\to \tilde{B}[z,dz]$ such that $h|_{z=0}=\tilde{f}$ and $h|_{z=1}=\tilde{g}$. The set of equivalence classes generated by the relation of free homotopy will be denoted by $[A,B]$.
\end{defi}
\begin{rem}
It is clear that for two formal cdgas $A$ and $B$, any map $\tilde{A}\to \tilde{B}$ is induced by a map $A\to B$. Furthermore, a homotopy between two maps  $f,g:A\to B$ gives rise to a free homotopy, but not vice-versa: there may be maps in $\falg$ which are freely homotopic but \emph{not} homotopic.
\end{rem}
The following proposition shows that the for two formal cdgas $A,B$, the set $[A,B]$ has a homotopy invariant meaning.
\begin{prop}\label{red}
Let $A, A^\prime, B$ and $B^\prime$ be cofibrant formal cdgas with $A$ weakly equivalent to $A^\prime$ and $B$ weakly equivalent to $B^\prime$. Then there is a bijective correspondence $[A,B]\to [A^\prime, B^\prime]$.
\end{prop}
\begin{proof}
Since $A$ and $A^\prime$ are cofibrant and weakly equivalent they must be Sullivan homotopy equivalent, so that there are maps $A\to A^\prime$ and $A^\prime\to A$ which are homotopy inverse. This determines two mutually inverse maps
$[A,B]\leftrightarrows [A^\prime, B]$. Similarly there is a bijective correspondence
$[A^\prime, B]\leftrightarrows [A^\prime, B^\prime]$.
\end{proof}
\begin{rem}
It is possible that the statement of the above proposition holds without the condition that $A$ and $A^\prime$ be cofibrant (compare Corollary \ref{hin}, (2)). To address this question properly one would have to go beyond the category $\falg$, allowing cdgas with more than one maximal ideal (equivalently considering not necessarily cocomplete dg coalgebras). The version stated here is sufficient for most applications, however.
\end{rem}
\section{The Quillen-Neisendorfer closed model category structure}
There is another, apparently, more natural, notion of a weak equivalence on the category of formal cdgas due to Quillen \cite{Q'} and Neisendorfer \cite{N} (in fact, Quillen and Neisendorfer worked with coalgebras but we saw that this difference is immaterial). Namely, we say that two formal cdgas $A$ and $B$ are \emph{homologically equivalent} if there is a (continuous) morphism $A\to B$ inducing an isomorphism on cohomology. In other words, our new weak equivalences are now (continuous) quasi-isomorphisms.

It is clear that weakly equivalent formal cdgas are homologically equivalent. Indeed, if $f:A\to B$ is a weak equivalence, then (by definition) the map of dglas $\L(f):\L(B)\to\L(A)$ is a quasi-isomorphism and then the map of formal cgdas $\C\L(f):\C\L(A)\to\C\L(B)$ will be a quasi-isomorphism from which it follows that $f$ was a quasi-isomorphism to begin with. On the other hand, a homology equivalence in $\falg$ need not be a weak equivalence; e.g. for any semisimple Lie algebra $\g$ the (formal) cdga $\C(\g)$ is quasi-isomorphic to its cohomology ring (which is the symmetric algebra on a collection of odd generators) but this quasi-isomorphism is not a weak equivalence because otherwise $\g$ would be quasi-isomorphic, and hence, isomorphic, to an abelian Lie algebra.

It makes sense, therefore, to consider the Bousfield localization of Hinich's closed model category $\falg$ with respect to homology equivalences. We will not treat this problem in full generality but note that the solution for the subcategory $\falg_+$ consisting of \emph{connected} (i.e. positively-graded) formal cdgas was given in \cite{N}, generalizing the previous work of Quillen, \cite{Q'}. Note that the functor $\L:\falg\mapsto\L ie$ restricts to the functor (which we denote by the same symbol) $\L:\falg_+\mapsto\L ie_+$ where $\L ie_+$ stands for the category of non-negatively graded dglas. Similarly the functor $\C:\L ie\mapsto\falg$ restricts to $\C:\L ie_+\mapsto \falg_+$.
\begin{theorem}\
\begin{enumerate}
\item
The category $\falg_+$ admits a closed model category structure such that fibrations are surjective maps and weak equivalences are homology equivalences of connected formal cdgas.
\item
The adjoint functors $\C$ and $\L$ induce an equivalence between the homotopy category of $\falg_+$ and the full subcategory of the homotopy category of $\L ie_+$ consisting of dglas whose homology are pro-nilpotent graded Lie algebras.
\end{enumerate}
\end{theorem}
\begin{proof}
The statement of the theorem is just a reformulation of Proposition 5.2 and Proposition 7.2 of \cite{N}.

\end{proof}
Since a weak equivalence of formal cdgas is stronger than a homology equivalence there are fewer cofibrant objects in the Quillen-Neisendorfer structure on $\falg_+$ compared to Hinich's one. For example, the formal cdga $\C(\g)$ will not be Quillen-Neisendorfer cofibrant unless $\g$ has pro-nilpotent homology Lie algebra; more generally a formal cdga $\hat{S}\Sigma^{-1} V^*$ representing an $L_\infty$ algebra $V$ will not be cofibrant unless a suitable nilpotency condition is satisfied. We restrict ourselves by describing a special kind of Quillen-Neisendorfer cofibrant objects coming from dglas with nilpotent homology; a more general result is contained in \cite{N}. Before formulating it recall the notion of a Sullivan minimal cdga.
\begin{defi}
A Sullivan minimal cdga is a cdga of the form $(S(W),d)$ where $W$ is a positively graded dg vector space which is a union of its subspaces $W=\bigcup W_i, i=1,2,\ldots$ with $d(W_i)\subset S(W_{i-1})$ and $d$ is decomposable: $d(W)\subset S(W)_+\cdot S(W)_+$.
\end{defi}
\begin{rem}\label{uncompleted}
Note that a Sullivan minimal cdga can always be completed and thus considered as a formal cdga (which represents a certain $L_\infty$ algebra).
We will be mostly interested in the case when $W$ is of finite type; in this case we will refer to $(S(W),d)$ as a Sullivan cdga of finite type. Note that in that case there is no difference between completed and uncompleted symmetric algebra. This is the most pleasant case: a Sullivan minimal cdga of finite type represents a cofibrant object both in the category $\falg$ and $\falg_+$.
\end{rem}
\begin{prop}
Let $\g$ be a non-negatively graded dgla with $H(\g)$ nilpotent and of finite type. Then the formal cdga $\C(\g)=\hat{S}\Sigma^{-1} \g^*$ is Quillen-Neisendorfer cofibrant. Moreover, there is an isomorphism in $\falg_+$:
\[
\C(\g)\cong A\otimes B
\]
where $A$ is a Sullivan minimal cdga and $B=\hat{S}\Sigma^{-1} V^*$ represents a linear contractible $L_\infty$ algebra $V$.
\end{prop}
\begin{proof}
This is essentially Proposition 3.12 of \cite{N} and we outline how its proof is adopted to the present context. Consider $\C(\g)$ as a formal cdga representing an $L_\infty$ algebra; then by the decomposition theorem there is an isomorphism of formal cdgas $\C(\g)\cong \hat{S}\Sigma^{-1} [H(\g)]^*\otimes B$ where $B$ represents a linear contractible $L_\infty$ algebra and $\hat{S}\Sigma^{-1} [H(\g)]^*$ represents the $L_\infty$ minimal model of the dgla  $\g$. Note that since $H=H(\g)$ is a graded Lie algebra of finite type the completed symmetric algebra on $\Sigma^{-1} [H(\g)]^*$ is the same as the uncompleted one.

Set $H(m,0):=\bigoplus_{i=0}^m H_i$, the $m$'th `Postnikov stage' of $H$. Consider further $H(m,1):=[H_0,H(m,0)]$, $H(m,2)=[H_0,H(m,1)]$ etc. so that $H(m,0)\supset H(m_1)\supset\ldots\supset H(m,n)$ is a finite filtration on $H(m,0)$ for every $m$. Then we have the refined filtration on $H$ (which corresponds to the principal refinement of the Postnikov tower of the rational space corresponding to $H$):
\[
H(0,n_0)\subset\ldots \subset H(0,0)=H_0\subset H(1, n_1)\subset\ldots \subset H(1, 0)=H_0\oplus H_1\subset\ldots
\]
and the induced filtration on $S\Sigma^{-1} H^*$ will satisfy the second condition of the Sullivan minimal model by the nilpotency of $H$.
\end{proof}
\begin{rem}
In the proposition above it was essential that $\g$ (and so $H=H(\g)$) was non-negatively graded. Without this assumption the nilpotency of $H$ alone does not even guarantee that the formal cdga representing an $L_\infty$ minimal model of $\g$ would be a polynomial algebra, let alone that it would be a minimal Sullivan cdga.
\end{rem}
\section{The MC moduli set and its homotopy invariance}
We will now describe the notion of an MC moduli set in the context of $L_\infty$ algebras and formal cdgas. 
\begin{defi} Let $(V,m)$ be an $L_\infty$ algebra and $A$ be a formal cdga. The formal cdga $A\otimes \hat{S}\Sigma^{-1} V^*$ represents an $A$-linear $L_\infty$ algebra by $A$-linear extension of scalars; the corresponding $L_\infty$ structure will be denoted by $m^A$. Then an element $\xi\in (A\otimes\Sigma V)_0$ is
\emph{Maurer-Cartan} (MC for short) if $(d_{A\otimes\Sigma V})(\xi)+\sum_{i=2}^\infty
\frac{1}{i!}m^A_i(\xi^{\otimes i})=0$. The set of Maurer-Cartan elements in $A\otimes\Sigma V$ will be
denoted by $\MC(V,A)$.

\end{defi}
 The correspondence $(V,A)\mapsto \MC(V,A)$ is clearly
functorial in $A$ and $V$. It is straightforward to check and well-known (see, e.g. \cite{CL}) that as a functor of the second argument $\MC(V,A)$ is representable by the formal cdga $\hat{S}\Sigma^{-1} V^*$. In other words, there is a natural bijective correspondence
\[
\MC(V,A)\cong \Hom_{\falg}(\hat{S}\Sigma^{-1} V^*, A).
\]
Furthermore, if $V$ is a dgla then we also have the following natural bijection, \cite{Q'}:
\[
\MC(V,A)\cong \Hom_{\L ie}(\L(A), V).
\]
If $\xi\in\MC(V,A)$ then we will abuse the notation and denote the corresponding map of formal cdgas $\hat{S}\Sigma^{-1} V^*\to A$ by the same symbol $\xi$.
\begin{defi}
Two MC elements $\xi,\eta\in\MC(V,A)$ are called \emph{homotopic} if there exists
an MC-element $h\in(V, A[z,dz])$ such that $h|_{z=0}=\xi$ and $h|_{z=1}=\eta$.
\end{defi}
\begin{rem}
Note that there is a slight imprecision built into this definition: an MC element as we defined it has coefficients in a formal cdga, but
$A[z,dz]$ is not formal. Nevertheless, the definition of an MC element does make sense for such a coefficient cdga and this will not cause us any further problems. This issue has already come up in the definition of a Sullivan homotopy for maps between formal cdgas. It is likely that
this minor nuisance can be fixed either by extending the category of formal cdgas suitably or by modifying the notion of the Sullivan homotopy, however we will refrain from elaborating on this further.
\end{rem}
\begin{prop}
The relation of homotopy on $\MC(V,A)$ is an equivalence relation. The set of equivalence classes under this relation is called
\emph{the MC moduli set} of $V$ with coefficients in $A$ and it will be denoted by $\MCmod(V,A)$.
\end{prop}
\begin{proof}
One only has to note that a homotopy between two MC elements $\xi,\eta\in \MC(V,A)$ is the same as a Sullivan homotopy between the maps
$\xi, \eta:\hat{S}\Sigma^{-1} V^*\to A$; the statement then follows from Corollary \ref{hin}, (1).
\end{proof}
The following is one of the main results about the MC moduli set. This result has a long history; a version of it is contained in the unpublished manuscript of Schlessinger and Stasheff \cite{SS}; it was further elaborated in \cite{GM}. Kontsevich formulated it using $L_\infty$ algebras in \cite{Kon} and Keller's method \cite{Kel} is essentially the same as ours. We believe that our formulation is the most general among those currently in existence; its obvious $\Z/2$-graded analogue is also valid.
\begin{theorem}\label{formdef}
Let $V$ be an $L_\infty$ algebra and $A$ be a formal cdga. Then for any $L_\infty$ algebra $U$ which is $L_\infty$ quasi-isomorphic to $V$ and any formal cdga $B$ which is weakly equivalent to $A$ there is a natural bijective correspondence
\[
\MCmod(V,A)\cong\MCmod(U,B).
\]
\begin{proof}
Let $\hat{S}\Sigma^{-1} V^*$ and $\hat{S}\Sigma^{-1} U^*$ be the formal cdgas which represent the $L_\infty$ algebras $V$ and $U$ respectively. Then we have $\MCmod(V,A)\cong [\hat{S}\Sigma^{-1} V^*, A]_*$ and $\MCmod(U,A)\cong [\hat{S}\Sigma^{-1} U^*, B]_*$ so the statement follows from Corollary \ref{hin}, (2).
\end{proof}
\end{theorem}
\begin{rem}\label{SS} In applications it often happens that the $L_\infty$ algebra $V$ that figures in Theorem \ref{formdef} is in fact a pronilpotent dgla in which case
the two MC elements are equivalent if and only if they are \emph{gauge equivalent}, an important result due to Schlessinger-Stasheff \cite{SS}, cf. also \cite{CL'} for a discussion of the notion of a gauge equivalence and a short proof.
\end{rem}
Sometimes one wants to consider MC elements in $V$ with coefficients in a not necessarily formal cdga $A$ (such as the ground field $\ground$, for instance). In such a situation the immediate problem is that the series $(d_{A\otimes\Sigma V})(\xi)+\sum_{i=2}^\infty
\frac{1}{i!}m^A_i(\xi^{\otimes i})=0$ defining the MC element $\xi$ need not converge. One important case when this problem goes away is when $V$ is a strict dgla. In that case the MC equation becomes the familiar equation of a flat connection:
\[
d(\xi)+\frac{1}{2}[\xi,\xi]=0.
\]
One can still call the elements $\xi\in A\otimes \Sigma V$ satisfying the above equation the MC elements of $V$ with coefficients in $A$ and denote the corresponding set by $\MC(V,A)$. Furthermore, the notion of a homotopy (or a gauge equivalence) makes perfect sense and we are entitled to form the moduli set $\MCmod(V,A)$ just as above; in the case when $V\otimes A$ is not a nilpotent dgla the homotopy may not be transitive and we have to take its transitive closure.  However the issue of homotopy invariance of $\MCmod(V,A)$ is more subtle; e.g. it is no longer true that an $L_\infty$ quasi-isomorphism in $V$ induces a bijection on the MC moduli set, even when $V$ is nilpotent. Here is a simple (in some sense simplest) example. Let $V$ be the 2-dimensional dgla generated by one element $x$ in degree $-1$ with $d(x)=-\frac{1}{2}[x,x]$. Clearly $V$ is acyclic, that is to say it is quasi-isomorphic to the zero Lie algebra. However it is easy to see that $\MCmod(V,\ground)$ consists of two elements -- $0$ and $x$. This demonstrates that one cannot simply remove the condition that the coefficient cdga $A$ be formal.

In our applications we will not consider MC elements with coefficients in a non-formal cdga, however we will need a version of the MC moduli set which could be viewed as a small step away from the formal situation. As far as we are aware this notion has not been considered before.
\begin{defi}
Let $(V,m)$ be an $L_\infty$ algebra with representing formal cdga $\hat{S}\Sigma^{-1} V^*$ and $A$
be a formal cdga. Then two $MC$ elements $\xi,\eta\in\MC(V,A)$ are called \emph{freely homotopic} if they are freely homotopic as maps of
formal cdgas $\hat{S}\Sigma^{-1} V^*\to A$.
Denote the transitive closure of the relation
of being freely homotopic by $\sim$; then the
\emph{reduced} MC moduli space of $V$ with coefficients in $A$ is defined as $\widetilde{\MCmod}(V,A):=\MC(V,A)/\sim$.
\end{defi}
The following result is a reformulation of Proposition \ref{red}.
\begin{prop}
Let $V$ be an $L_\infty$ algebra and $A$ be  cofibrant formal cdga. Then the set $\widetilde{\MCmod}(V,A)$ is homotopy invariant in the sense that
for any $L_\infty$ algebra $V^\prime$ quasi-isomorphic to $V$ and a cofibrant formal cdga $A^\prime$ quasi-isomorphic to $A$ there is a bijective correspondence
\[
\widetilde{\MCmod}(V,A)\cong \widetilde{\MCmod}(V^\prime,A^\prime).
\]
\end{prop}
\noproof
\begin{rem}
We see that the relation on the set of MC elements of being freely homotopic is stronger then the usual homotopy relation. It follows that there is a surjective map ${\MCmod}(V,A)\to \widetilde{\MCmod}(V,A)$. A more precise relationship between these two moduli sets can sometimes be established; in fact it follows from our topological interpretation given in the last section that in the case when the representing formal cdga of $V$ is a Sullivan model of a rational space $X$ there is an action of the group $\pi_1(X)$ on $\MCmod(V,A)$ so that the quotient by this action is $\widetilde{\MCmod}(V,A)$. In the context of deformation theory we can offer the following analogy: given an object $O$ and its universal deformation with a (perhaps dg) base $X$ there is a `residual' action of the group $\Aut(O)$ of automorphisms of $O$ on $X$. In favorable cases one expects the quotient to be the germ of the moduli space at $O$.
\end{rem}
\subsection{Connected covers of dglas and MC moduli}
For the needs of rational homotopy theory we need to consider connected formal cdgas (such a cdga $A$ is required to satisfy $A^i=0$ for $i\leq 0$) and connected $L_\infty$ algebras (such an $L_\infty$ algebra $V$ is required to satisfy $V_i=0$ for $i<0$).
\begin{defi}
Let $(V,m_V)$ be an $L_\infty$ algebra; then its \emph{connected cover} is the $L_\infty$ algebra $V\langle 0\rangle$ defined by the formula
\[V\langle 0\rangle_i=\begin{cases}V_i, \text{~if~ } i>0\\ \Ker\{d:V_0\to V_{-1}\} \text{~if~ } i=0\\ 0, \text{~if~} i<0\end{cases}.\]
The $L_\infty$ structure $m_{V\langle 0\rangle}$ on $V\langle 0\rangle$ is the obvious restriction of $m_V$.
\end{defi}
It is clear that there is a natural (strict) $L_\infty$ map $V\langle 0\rangle\to V$, moreover for two $L_\infty$ quasi-isomorphic $L_\infty$ algebras $V$ and $U$ the corresponding connected covers $V\langle 0\rangle$ and $U\langle 0\rangle$ are $L_\infty$ quasi-isomorphic.
\begin{prop}\label{conn}
For any connected formal cdga $A$ and an $L_\infty$ algebra $V$ there is a natural bijection
\[
\MCmod(V,A)\cong\MCmod(V\langle 0\rangle,A).
\]
\end{prop}
\begin{proof}
Without loss of generality we can assume that $V$ is a minimal $L_\infty$ algebra. Choose a homogeneous basis $\{x_i\}, \{y_i\}$ in $V$ for which $|x_i|\geq 0$ and $|y_i|<0$. Passing to the dual we see that the $L_\infty$ algebra $V$ is represented by a formal cdga $\hat{S}\{\Sigma^{-1} x_i^*,\Sigma^{-1} y_i^*\}$ whereas the  $L_\infty$ algebra $V\langle 0\rangle$ is represented by its quotient formal cdga $\hat{S}\{\Sigma^{-1} x_i^*\}$. Since MC elements in $V$ and $V\langle 0\rangle$ with coefficients in $A$ are represented by maps from $\hat{S}\{\Sigma^{-1} x_i^*,\Sigma^{-1} y_i^*\}$ and $\hat{S}\{\Sigma^{-1} x_i^*\}$ into $A$ the degree considerations give a bijection $\MC(V,A)\cong\MC(V\langle 0\rangle,A)$. It is likewise clear that any homotopy of elements in $\MC({V},A)$ lifts to a homotopy in $\widetilde{MC}(V,A)$ from which the desired statement follows.
\end{proof}
\section{Maurer-Cartan twisting}
We now recall the procedure of \emph{twisting} in $L_\infty$ algebras by MC elements, following \cite{CL'}. Let $(V,m)$ be an $L_\infty$ algebra, $A$ be a formal cdga and $\xi\in\MC(V,A)$.  We may regard
$\xi$ as a (formal) $A$-linear derivation of the formal cdga
$A\otimes \hat{S}\Sigma^{-1} V^*$; indeed $\xi$ could be viewed as a linear
function on $\Sigma^{-1} V^*\cong 1\otimes \Sigma^{-1} V^*\subset
A\otimes \hat{S}\Sigma^{-1} V^*$ and we extend it to the whole of $A\otimes \hat{S}\Sigma^{-1} V^*$ by the $A$-linear Leibniz rule.

Exponentiating the derivation $\xi$ we obtain an automorphism of $A\otimes \hat{S}\Sigma^{-1} V^*$ so that $e^\xi:=\id+\xi+\frac{\xi^2}{2!}+\ldots$; the convergence of this formal power series is ensured by the formality of $A$. Then it is straightforward to see that the MC condition on $\xi$ implies (in fact, is equivalent to) that the formal derivation $e^\xi (d_{A\otimes\Sigma V}+m^A) e^{-\xi}=e^\xi m^Ae^{-\xi}-d_{A\otimes \Sigma V}(\xi)$ is an $L_\infty$ structure on $A\otimes \hat{S}\Sigma^{-1} V^*$, in other words, that it has no constant term. We will write this new $A$-linear $L_\infty$ structure on $V$ as $((A\otimes V)^\xi,m^\xi)$ and call it the \emph{twisting} of $m$ by $\xi$. 
Note that the formal derivation $e^\xi m^Ae^{-\xi}-d_{A\otimes \Sigma V}(\xi)$ is also defined on $\tilde{A}\otimes \hat{S}\Sigma^{-1}
 V^*$ and thus, we can define the unital version of the MC twisting: $(\tilde{A}\otimes V)^\xi$.
The following result shows that homotopic $MC$ elements determine weakly equivalent twistings.
\begin{prop}\label{hom1}
Let $\xi, \eta$ be two homotopic elements in $\MC(V,A)$ where $(V,m)$ is an $L_\infty$ algebra and $A$ is a formal cdga. Then:
\begin{enumerate}
\item  The $L_\infty$ algebras $(\tilde{A}\otimes V)^\xi$ and $(\tilde{A}\otimes V)^\eta$ are $L_\infty$ isomorphic.
\item The $L_\infty$ algebras $(A\otimes V)^\xi$ and $(A\otimes V)^\eta$ are $L_\infty$ isomorphic.
\end{enumerate}
\end{prop}
\begin{proof}
Let us prove the first statement; the proof of the second statement is similar. A homotopy between $\xi$ and $\eta$ is an MC element $h\in \MC(\tilde{A}\otimes V[z,dz])$ that specializes to $\xi$ and $\eta$ at $z=0$ and $z=1$ respectively. Consider the $L_\infty$ algebra $(\tilde{A}\otimes V[z,dz])^h$, the twisting of $\tilde{A}\otimes V[z,dz]$ by $h$. This $L_\infty$ algebra structure can be viewed as an element in $\MC(\g, A[z,dz])$ where $\g$ is the Lie algebra of formal derivations of $\hat{S}\Sigma^{-1} V^*$ having no constant and linear terms. In other words, $h$ is a homotopy between two elements in $\MC(\g,\tilde{A})$. Such a homotopy implies gauge equivalence (see Remark \ref{SS}) and it follows that the $L_\infty$ algebras $(A\otimes V)^\xi$ and $(A\otimes V)^\eta$ are $L_\infty$ isomorphic as desired.
\end{proof}
\begin{rem}
Let $V$ be a dgla (as opposed to an $L_\infty$ algebra), not necessarily nilpotent. Then one can define its twisting without tensoring with a formal cdga; namely for $\xi\in\MC(V):=\MC(V,\ground)$ we get a twisted dgla $V^\xi$ whose Lie bracket is the same as that of $V$ and the differential is given as $d^\xi(v)=d(v)+[\xi,v]$ for $v\in V$. Then it is easy to see that in this situation the proof of Proposition \ref{hom1} still implies that two homotopic MC elements give rise to $L_\infty$ isomorphic twisted dglas. However the next result uses the formal cdga variable in an essential way.
\end{rem}
\begin{prop}\label{hom2}\
\begin{enumerate}
\item
Let $f:(V,m_V)\to (U,m_U)$ be an $L_\infty$  quasi-isomorphism, $A$ be a formal cdga, $\xi\in\MC(V,A)$ and $f_*(\xi)$ be the element in $\MC(U,A)$ corresponding to $\xi$ under the map $f_*:\MC(V,A)\to\MC(U,A)$ induced by $f$. Then there are natural $L_\infty$ quasi-isomorphisms $f^\xi:(A\otimes V)^\xi\to(A\otimes U)^{f_*(\xi)}$ and $\tilde{f}^\xi:(\tilde{A}\otimes V)^\xi\to(\tilde{A}\otimes U)^{f_*(\xi)}$
\item
Let $(V,m)$ be an $L_\infty$ algebra, $g:A\to B$ be a weak equivalence between two formal cdgas $A$ and $B$, $\xi\in\MC(V,A)$ and $g_*(\xi)$ be the element in $\MC(V,B)$ corresponding to $\xi$ under the map $g_*:\MC(V,A)\to\MC(V,B)$ induced by $g$. Then there are natural $L_\infty$ quasi-isomorphisms $g^\xi:(A\otimes V)^\xi\to(B\otimes V)^{g_*(\xi)}$ and  $\tilde{g}^\xi:(\tilde{A}\otimes V)^\xi\to(\tilde{B}\otimes V)^{g_*(\xi)}$.
\end{enumerate}
\end{prop}
\begin{proof}
We restrict ourselves with proving the statements involving $f^\xi$ and $g^\xi$; the proofs for $\tilde{f}^\xi$ and $\tilde{g}^\xi$ are completely analogous. For (1) set $\tilde{m}_V:=d_{A\otimes\Sigma V}+ m^A_V$ and  denote by $\tilde{m}^\xi_V$ the formal $A$-linear derivation of $A\otimes \hat{S}\Sigma^{-1} V^*$ obtained by
twisting $\tilde{m}_V$ with $\xi$ so that $\tilde{m}^\xi_V=e^\xi\tilde{m}^A_V e^{-\xi}$; similarly set
$\tilde{m}_U:=d_{A\otimes\Sigma U}+ m^A_U$ and $\tilde{m}^\xi_U=e^{f_*(\xi)}\tilde{m}^A_U e^{-f_*(\xi)}$. Then the map $f^\xi:=e^{f_*(\xi)}fe^{-\xi}:\hat{S}\Sigma^{-1} U^*\to \hat{S}\Sigma^{-1} V^*$ determines an $L_\infty$ map $(A\otimes V)^\xi\to (A\otimes U)^{f_*(\xi)}$. Indeed,
\begin{align*}
f^\xi \tilde{m}^\xi_V=&[e^{f_*(\xi)}fe^{-\xi}][e^\xi \tilde{m}^A_V e^{-\xi}]\\
=&e^{f_*(\xi)}f  \tilde{m}^A_V e^{-\xi}\\
=&e^{f_*(\xi)} \tilde{m}^A_U f  e^{-\xi}\\
=&[e^{f_*(\xi)} \tilde{m}^A_U e^{-f_*(\xi)}][ e^{f_*(\xi)} f  e^{-\xi}]\\
=&\tilde{ m}^{f_*(\xi)}_U  f^\xi.
\end{align*}
Here we used the equality $f\tilde{ m}^A_V=\tilde{m}^A_Uf$ which holds since $f$ is an $L_\infty$ map.

Further, the 1-component of the map $f^\xi$ is a map between filtered dg vector spaces
\[f_1^\xi:A\otimes V\to A\otimes U\]
where the filtration is induced by the canonical multiplicative filtration on the formal cdga $A$.
It is clear that on the level of the associated graded the map $f^\xi_1$ reduces to $f_1$, the first
component of the original $L_\infty$ map between $V$ and $U$. Since the latter is a quasi-isomorphism
we conclude that $f_1^\xi$ is likewise a quasi-isomorphism as desired.

Now let us prove (2). Using $(1)$ and the fact that any $L_\infty$ algebra is $L_\infty$ quasi-isomorphic to a strict dgla we reduce the statement to the case when $V$ is a strict dgla. Furthermore, without loss of generality we can assume that $A\to B$ is a filtered map (with respect to some admissible filtrations on $A$ and $B$), cf. Remark \ref{filt}. Let $\xi\in\MC(V,A)$; the twisted differential on $(A\otimes V)^\xi$ will have the form $d^\xi=d_{V\otimes A}+[\xi,?]$ where $d_{V\otimes A}$ is the untwisted differential on $A\otimes V$. It follows that there is an isomorphism of associated graded dglas:
\[
\Gr(A\otimes V)^\xi\cong\Gr(A\otimes V)
\]
since the `twisted' part $[?,\xi]$ of the differential $d^\xi$ vanishes upon passing to the associated graded dglas. Similarly
\[
\Gr(B\otimes V)^{g_*\xi}\cong\Gr(B\otimes V)
\]
It follows that the map $g\otimes\id:A\otimes V\to B\otimes V$ induces a quasi-isomorphism
\[
\Gr(A\otimes V)^\xi\to
\Gr(B\otimes V)^{g_*\xi}
\]
and so $g\otimes\id$ must be a quasi-isomorphism.
\end{proof}
\begin{rem}
If an $L_\infty$ map $f:(V,m_V)\to(U,m_U)$ is not a quasi-isomorphism and $\xi\in\MC(V,A)$ then there is still an $L_\infty$ map $f^\xi:(A\otimes V)^\xi\to(A\otimes U)^{f_*(\xi)}$ given by the same formula $f^\xi:=e^{f_*(\xi)}fe^{-\xi}$; of course it need not be an $L_\infty$ quasi-isomorphism. It is straightforward to write $f^\xi$ in components; one has for $x_i\in \Sigma^{-1} V, i=1,2\ldots$:
\[
f^\xi_n(x_1,\ldots, x_n)=\sum_{i=0}^\infty \frac{1}{i!}f_{n+i}(\xi,\ldots, \xi,x_1,\ldots, x_n).
\]
We will not need this explicit formula.
\end{rem}
Combining Propositions \ref{hom1} and \ref{hom2} we obtain the following result.
\begin{cor}
Let $f:(V,m_V)\to(U,m_U)$ be an $L_\infty$ quasi-isomorphism and $A\to B$ be a weak equivalence between formal cdgas $A$ and $B$. For $\xi\in\MCmod(V,A)$ denote by $\tilde{\xi}\in\MCmod(U,A)$ the equivalence class corresponding to $\xi$ under the bijection $\MCmod(V,A)\cong\MCmod(U,B)$ induced by $f$ and $g$. Then the $L_\infty$ algebras $(A\otimes V)^\xi$ and $(B\otimes U)^{\tilde{\xi}}$ are $L_\infty$ quasi-isomorphic.
\end{cor}\noproof
\section{Examples of twisting: Chevalley-Eilenberg and Harrison cohomology}
In this section we explain how to treat Chevalley-Eilenberg cohomology of dglas and Harrison-Andr\'e-Quillen cohomology of cdgas as instances of MC twisting. First let us recall the standard definitions of Chevalley-Eilenberg and Harrison cohomology, cf. for example \cite{Loday, BL}.
\begin{defi}\label{ungraded}\
\begin{itemize}
\item
Let $V$ be a Lie algebra and $M$ be a $V$-module. Then the Chevalley-Eilenberg complex of $V$ with coefficients in $M$ is defined as
$C^n_{\CE}(V,M)\subset \Hom(V^{\otimes n}, M)$ consisting of skew-symmetric multilinear functions on $V$ with values in $M$. The differential
$C^n_{\CE}(V,M)\to C^{n+1}_{\CE}(V,M)$ is defined as follows:
\begin{align}\label{CE}
(df)(v_1,\ldots, v_{n+1})&=\sum_{1\leq i<j\leq n+1}(-1)^{i+j-1}f([v_i,v_j],v_1,\ldots,\hat{v}_i,\ldots \hat{v}_j,\ldots v_{n+1})\\
\nonumber &+\sum_{i=1}^{n+1}(-1)^iv_if(v_1,\ldots,\hat{v}_i,\ldots, v_{n+1}).
\end{align}
\item
Let $A$ be a commutative algebra and $M$ be an $A$-module. Then the Harrison-Andr\'e-Quillen complex (or simply Harrison for short) of $A$ with coefficients in $M$ is defined as $C^n_{\Harr}(A,M)\subset \Hom(A^{\otimes n}, M)$ where $f:A^{\otimes n}\to M$ belongs to $C^n_{\Harr}(A,M)$ if $f$ vanishes on all shuffle products of elements in $A^{\otimes k}$ and $A^{\otimes i}$ with $k+i=n$. The differential $C^n_{\Harr}(A,M)\to C^{n+1}_{\Harr}(A,M)$ is the restriction of the usual Hochschild differential:
\begin{align}\label{AQ}
(df)(a_1,\ldots, a_{n+1})=&a_1f(a_2,\ldots, a_{n+1})\\ \nonumber+&\sum_{k=1}^n(-1)^kf(a_1,\ldots,a_ka_{k+1},\ldots, a_{n+1})+(-1)^{k+1}f(a_1,\ldots, a_n)a_{n+1}.
\end{align}
\end{itemize}
\end{defi}
We will need a more general type of the Chevalley-Eilenberg cohomology when $V$ is a dgla or even an $L_\infty$ algebra.
The version of the Harrison cohomology we require will be for $A$ and $B$ being formal cdgas. The module of coefficients will also be assumed to be dg; additionally it will carry the structure of a dgla (or $L_\infty$) algebra in the Chevalley-Eilenberg case and that of a formal cdga in the Harrison case. In this more general setup it is much more convenient to package the Chevalley-Eilenberg and Harrison complexes within the formalism of MC twistings.
\begin{defi}\label{graded}\
\begin{itemize}
\item
Let $(V,m_V)$ be an $L_\infty$ algebra representable by the formal cdga $\hat{S}\Sigma^{-1} V^*$, $(U, m_U)$ be another $L_\infty$ algebra and $\xi\in \MC(U,\hat{S}\Sigma^{-1} V^*)$ be the MC element corresponding to an $L_\infty$ map $V\to U$. Then the Chevalley-Eilenberg complex of $V$ with coefficients in $U$ is defined as
\[
C_{\CE}^\xi(V,U):=(\hat{S}\Sigma^{-1} V^*\otimes U)^\xi
\]
\item
Let $A$ and $B$ be formal cdgas and $\xi\in MC(\L(A),B)$ be the MC element corresponding to a map of formal cdgas $A\to B$.  Then the Harrison complex of $A$ with coefficients in $B$ is defined as
\[
C_{\Harr}^\xi(A,B):=(\tilde{B}\otimes \L(A))^\xi.
\]
\end{itemize}
\end{defi}
\begin{rem}
The reader will note the apparent asymmetry in our exposition of Chevalley -Eilenberg and Harrison cohomologies; indeed it is possible to treat the latter for not necessarily formal cdgas or even $C_\infty$ algebras, cf. \cite{HL} for a detailed discussion of $C_\infty$ algebras. The reason for this asymmetry is that the present framework is adequate for the needs of rational homotopy theory; additionally $C_\infty$ algebras are more difficult to work with than $L_\infty$ algebras; particularly the notion of an MC element of a dgla with values in a $C_\infty$ algebra is combinatorially rather messy and we do not have clear applications justifying this added complexity.
\end{rem}
\begin{prop}
If $V$ and $U$ are Lie algebras and $\xi:V\to U$ is a Lie algebra map then Definition \ref{ungraded} (1) is equivalent to Definition \ref{graded} (1) so that there is an isomorphism
 \[
C_{\CE}^\xi(V,U)\cong C_{\CE}(V,U).
 \]
Similarly if $A$ and $B$ are (non-unital) finite dimensional nilpotent algebras and $A\to B$ is an algebra map then Definition \ref{ungraded} (2) is equivalent to Definition \ref{graded} (2) in the following sense
\[
C_{\Harr}^\xi(A,B)\cong C_{\Harr}(A,B).
\]

\end{prop}
\begin{proof}
Let us prove the first statement. Consider a map
\[
f:\Hom([T^n\Sigma^{-1} V],U)^{S_n}\to \Hom(T^n\Sigma^{-1} V, U)_{S_n}\cong [\hat{T}^n\Sigma^{-1} V^*]_{S_n}\otimes U\cong
\hat{S}^n\Sigma^{-1} V^*\otimes U,
\]
the natural isomorphism from $S_n$-invariants to $S_n$-coinvariants. The map $f$ identifies the $n$th cochains of the Lie algebra $V$ with coefficients in the Lie algebra $U$ in the sense of Definition \ref{ungraded} with $(\hat{S}^n\Sigma^{-1} V^*\otimes U)^\xi$. Recall that the differential in the dgla $(\hat{S}^n\Sigma^{-1} V^*\otimes U)^\xi$ has the form $[m_V,?]+[\xi,?]$; it is then straightforward to check that $[m_V,?]$ and $[\xi,?]$ correspond to the first and second terms in the formula (\ref{CE}) respectively.

For the second statement consider a map
\[
g:\Hom([T^n\Sigma^{-1} A],B)\to \hat{T}^n\Sigma^{-1} A^*\otimes B.
\]
Then the restriction of $g$ onto the space of those multilinear maps which vanish on shuffle products can be identified with $(\operatorname{Prim}\hat{T}\Sigma^{-1} A^*)\bigcap(\hat{T}^n\Sigma^{-1} A^*)\otimes B$ where $\operatorname{Prim}\hat{T}\Sigma^{-1} A^*$ is the space of primitive elements in the Hopf algebra $\hat{T}\Sigma^{-1} A^*$ (supplied with the standard cocommutative coproduct). This space  of primitive elements is further identified with the free Lie algebra on $\Sigma^{-1} A^*$; therefore $C^n_{\Harr}(A,B)$, the space of $n$th cochains of the algebra $A$ $V$ with coefficients in the Lie algebra $U$ in the sense of Definition \ref{ungraded} is isomorphic to $B\otimes \L ie(A)$. A simple check shows that the differential (\ref{AQ}) agrees with the differential in $[B\otimes \L ie(A)]^\xi$.
\end{proof}
\begin{rem}\
\begin{itemize}
\item
We see that in the situation of Definition
\ref{graded} \
the Chevalley-Eilenberg complex
$C_{\CE}(V,U)$ has the structure
of an $L_\infty$ algebra and the Harrison complex
$C_{\Harr}(A,B)$ has the structure of a dgla. According to our convention we should, therefore, view them as homologically graded dg spaces; however this would contradict with the traditional usage and so we retain the cohomological grading for these complexes.
\item
A natural question is whether one can extend our definitions to the case of an \emph{arbitrary module of coefficients}. We will now outline how this can be done in the Chevalley-Eilenberg case; the Harrison case could be dealt with similarly.

Suppose that $V$ is an $L_\infty$ algebra with the representing formal cdga $\hat{S}\Sigma^{-1} V^*$ and $M$ is an $L_\infty$ module over $V$; that means  that there is an $L_\infty$ map $f:V\to \End(M)$. Then we have the induced map
\[
f_*:\MC(V,\hat{S}\Sigma^{-1} V^*)\to \MC(\End(M),\hat{S}\Sigma^{-1} V^*).
\]
Consider the element $f_*(\xi)\in \MC(\End(M),\hat{S}\Sigma^{-1} V^*)$ where $\xi\in \MC(V,\hat{S}\Sigma^{-1} V^*)$ is the canonical MC element. The element $f_*(\xi)$ can be viewed
as an $\hat{S}\Sigma^{-1} V^*$-linear endomorphism of $\hat{S}\Sigma^{-1} V^*\otimes M$ and we define the Chevalley-Eilenberg complex of $V$ with coefficients in $M$ as
\[
C_{\CE}(V,M)=\hat{S}\Sigma^{-1} V^*\otimes M
\]
supplied with the twisted differential
$d^{f_*(\xi)}(a\otimes m)=d(a\otimes m)+f(a\otimes m)$ for $a\in\hat{S}\Sigma^{-1} V^*, m\in M$.
\end{itemize}
\end{rem}
From now on we will omit the superscript $\xi$  in the symbols for the Harrison and Chevalley-Eilenberg cohomology whenever the choice of the corresponding MC element is clear from the context. Additionally, we introduce the notion of \emph{truncated} Chevalley-Eilenberg and Harrison cohomology.
\begin{defi}\
\begin{itemize}
\item
Let $(V,m_V)$ be an $L_\infty$ algebra representable by the formal cdga $\hat{S}\Sigma^{-1} V^*$, $(U, m_U)$ be another $L_\infty$ algebra and $\xi\in \MC(U,\hat{S}\Sigma^{-1} V^*)$ be the MC element corresponding to an $L_\infty$ map $V\to U$. Then the truncated Chevalley-Eilenberg complex of $V$ with coefficients in $U$ is defined as
\[
\overline{C}_{\CE}(V,U):=([\hat{S}\Sigma^{-1} V^*]_+\otimes U)^\xi.
\]
\item
Let $A$ and $B$ be formal cdgas and $\xi\in MC(\L(A),B)$ be the MC element corresponding to a map of formal cdgas $A\to B$.  Then the truncated Harrison complex of $A$ with coefficients in $B$ is defined as
\[
\overline{C}_{\Harr}(A,B):=(B\otimes \L(A))^\xi.
\]
\end{itemize}
\end{defi}
It is clear that there $\overline{C}_{\CE}(V,U)$ and $\overline{C}_{\Harr}(A,B)$ are sub dglas of $C_{\CE}(V,U)$ and $C_{\Harr}(A,B)$; moreover there are  the following short exact sequences of dg vector spaces:
\[
\overline{C}_{\CE}(V,U)\to C_{\CE}(V,U)\to U;\]
\[
\overline{C}_{\Harr}(A,B)\to C_{\Harr}(A,B)\to B.
\]

It turns out that and Chevalley-Eilenberg cohomology of $L_\infty$ algebras reduce to Harrison cohomology of formal cdgas.
\begin{prop}\label{CEH}
Let $(V,m_V)$ and $(U,m_U)$ be $L_\infty$ algebras with representing formal cdgas $B=\hat{S}\Sigma^{-1} V^*$ and $A=\hat{S}\Sigma^{-1} U^*$ respectively; let $\hat{S}\Sigma^{-1} U^*\to \hat{S}\Sigma^{-1} V^*$ be the map of formal cdgas representing an $L_\infty$ map $V\to U$. Then there are natural $L_\infty$ quasi-isomorphisms:
\[
C_{\CE}(V,U)\cong C_{\Harr}(A,B);
\]
\[
\overline{C}_{\CE}(V,U)\cong \overline{C}_{\Harr}(A,B);
\]
\end{prop}
\begin{proof}
We will restrict ourselves with proving the first isomorphism; the proof of the second one is similar. Denote by $\xi_1$ the MC element in $\hat{S}\Sigma^{-1} V^*\otimes U$ corresponding to the map $\xi:A=\hat{S}\Sigma^{-1} U^*\to B=\hat{S}\Sigma^{-1} V^*$; similarly denote by $\xi_2$ the MC element in $B\otimes \L(A)$ corresponding to $\xi$.
Recall that $C_{\CE}(V,U)=(\hat{S}\Sigma^{-1} V^*\otimes U)^{\xi_1}=(B\otimes U)^\xi$ and that $C_{\Harr}(A,B)=[B\otimes \L(A)]^{\xi_2}$. The adjunction morphism $\C\L(A)\to A$ is a weak equivalence of formal cdgas by Theorem \ref{Hinich} from which it follows that $\L(A)$ and $U$ are $L_\infty$ quasi-isomorphic. Note also that $\xi_2$ corresponds to $\xi_1$ under this $L_\infty$ quasi-isomorphism. Now the desired statement follows from Proposition \ref{hom2}.
\end{proof}
The Chevalley-Eilenberg and Harrison cohomology are homotopy invariant in the following sense.
\begin{prop}\
\begin{enumerate}
\item
Let $(V, m_V), (U,m_U)$ be two $L_\infty$ algebras with representing formal cdgas $\hat{S}\Sigma^{-1} V^*$ and $\hat{S}\Sigma^{-1} V^*$ respectively. Let $\xi,\eta:\hat{S}\Sigma^{-1} U^*\to \hat{S}\Sigma^{-1} V^*$ be two maps representing two homotopic $L_\infty$ maps $V\to U$. Then the $L_\infty$ algebras $C^\xi_{\CE}(V,U)$ and $C^\eta_{CE}(V,U)$ are $L_\infty$ isomorphic. Similarly the $L_\infty$ algebras $\overline{C}^\xi_{\CE}(V,U)$ and $\overline{C}^\eta_{CE}(V,U)$ are $L_\infty$ isomorphic.
\item
Let $A,B$ be two formal cdgas and $\xi,\eta:A\to B$ be two homotopic maps between them. Then the dglas $C^\xi_{\Harr}(A,B)$ and $C^\eta_{\Harr}(A,B)$ are $L_\infty$ isomorphic. Similarly the dglas $\overline{C}^\xi_{\Harr}(A,B)$ and $\overline{C}^\eta_{\Harr}(A,B)$ are $L_\infty$ isomorphic.
\end{enumerate}
\end{prop}
\begin{proof}
We restrict ourselves with proving the corresponding statements for the untruncated complexes; the proofs  for the truncated analogues are completely parallel. For (1) consider the MC elements $\xi,\eta\in\MC(U,\hat{S}\Sigma^{-1} V^*)$ associated with the corresponding $L_\infty$ maps. These MC elements are homotopic and thus, by Proposition \ref{hom1} the $L_\infty$ algebras $(\hat{S}\Sigma^{-1} V^*\otimes U)^\xi$ and $(\hat{S}\Sigma^{-1} V^*\otimes U)^\eta$ are $L_\infty$ isomorphic as required. For (2) viewing $\xi$ and $\eta$ as elements in $\MC(\L(A),B)$ and noting that these elements are homotopic, we conclude similarly that the dglas $[B\otimes\L(B)]^\xi$ and $[B\otimes\L(B)]^\eta$ are $L_\infty$ isomorphic.
\end{proof}
The following result is a direct consequence of Proposition \ref{hom2}:
\begin{prop}\
\begin{enumerate}
\item
Let $(V,m_V)$, $(V^\prime, m_{V^\prime})$, $(U,m_U)$ and $(U^\prime,m_U^\prime)$ be $L_\infty$ algebras, $(V,m_V)\to(V^\prime,m_{V^\prime})$ and $(U^\prime,m_{U^\prime})\to (U,m_U)$ be $L_\infty$ quasi-isomorphisms and $(V^\prime,m_{V^\prime})\to(U^\prime,m_U^\prime)$ be an $L_\infty$ map. Then there are natural $L_\infty$ quasi-isomorphisms
\[
C_{\CE}(V,U)\to C_{\CE}(V^\prime,U^\prime);~ \overline{C}_{\CE}(V,U)\to \overline{C}_{\CE}(V^\prime,U^\prime)
\]
where $C_{\CE}(V,U)$ is formed using the composition $L_\infty$ map $(V,m_V)\to(V^\prime,m_{V^\prime})\to (U^\prime,m_U^\prime)\to(U,m_{U})$.
\item
Let $A,A^\prime, B$ and $B^\prime$ be formal cdgas, $A\to A^\prime$ and $B^\prime\to B$ be weak equivalences and $A^\prime \to B^\prime$ be a map of formal cdgas. Then there are natural $L_\infty$ quasi-isomorphisms
\[
C_{\Harr}(A,B)\to C_{\Harr}(A^\prime,B);~\overline{C}_{\Harr}(A,B)\to \overline{C}_{\Harr}(A^\prime,B).
\]
where $C_{\Harr}(A,B)$ is formed using the composition $A\to A^\prime\to B^\prime\to B$.
\end{enumerate}
\end{prop}
\noproof
\begin{rem}
As an aside we mention that there is an analogue of the dgla structure on $C_{\Harr}$ or $C_{\CE}$ in the context of the \emph{Hochschild complex} of an $A_\infty$ algebra; for a detailed discussion of the latter concept see \cite{Keller} or \cite{HL}. Namely, suppose that $(V,m_V), (U,m_U)$ are two $A_\infty$ algebras with representing formal dgas $\hat{T}\Sigma^{-1} V^*$ and $\hat{T}\Sigma^{-1} U^*$ respectively; suppose further that $\xi:\hat{T}\Sigma^{-1} U^*\to \hat{T}\Sigma^{-1} V^*$
is a map representing an $A_\infty$ morphism $V\to U$. The latter morphism corresponds to an $A_\infty$ MC element $\xi\in\MC(U,\hat{T}\Sigma^{-1} V^*)$, cf. \cite{CL} concerning $A_\infty$ MC elements and associated twistings. Then $C_{\operatorname {Hoch}}(V,U)$, the Hochschild complex of $V$ with coefficients in $U$ can be naturally identified with $(\hat{T}\Sigma^{-1} V^*\otimes U)^\xi$. We conclude that $C_{\operatorname {Hoch}}(V,U)$ has the structure of an $A_\infty$ algebra. This structure has been studied in some detail in the case when $V=U$ and $\xi$ is the identity morphism, \cite{getjon} cf. in which case it is homotopy abelian as part of a richer structure of a $G_\infty$ algebra. In general this $A_\infty$ structure may not be homotopy abelian.
\end{rem}
\section{Rational homotopy of function spaces}
In this section we will use the developed technology of MC twistings to construct explicit rational models for function spaces. Let $X$ and $Y$ be two connected nilpotent rational $CW$ complexes of finite type. Additionally, we assume that either $X$ is a finite CW complex or $Y$ has a finite Postnikov tower; this condition ensures that the spaces of maps between  $X$ and $Y$ are homotopically equivalent to finite type complexes. We remark that the latter condition is not necessary, at least in the philosophical sense; we impose it only because the current state of rational homotopy theory does not provide a Lie-Quillen model for a nilpotent space that is simultaneously not simply-connected and not of finite type. We believe that such a model should exist.  Denote by $F(X,Y)$ ($F_*(X,Y)$) the spaces of continuous maps (based continuous maps) between $X$ and $Y$. We denote by $A(X)$ a Sullivan minimal model of $X$, as described in \cite{BG} and by $L(Y)=\L(A(X)_+)$ its Lie-Quillen model of $Y$, \cite{N}. Note that $A(X)_+$ can be viewed as a formal cdga since it is a symmetric algebra on finitely many generators in positive degrees.  Then we have the following result.
\begin{theorem}\label{main}\
\begin{enumerate}
\item
\begin{enumerate}
\item
There is a bijection $\pi_0F(X,Y)\cong \widetilde{\MCmod}(L(Y),A(X)_+)$.
\item
There is a bijection $\pi_0F_*(X,Y)\cong \MCmod(L(Y),A(X)_+)$.
\end{enumerate}
\item Let $\xi:X\to Y$ denote both a base point in $F(X,Y)$ and a representative of the corresponding element in $\MC(L(Y),A(X)_+)$.
Further, denote by $F^\xi(X,Y)$ and $F^\xi_*(X,Y)$ the connected component of $F(X,Y)$ and of $F_*(X,Y)$ containing $\xi$.
\begin{enumerate}
\item
The dgla
$[A(X)\otimes L(Y)]^\xi\langle 0\rangle$ is a Lie-Quillen model of $F^\xi(X,Y)$.
\item
The dgla
$[A(X)_+\otimes L(Y)]^\xi\langle 0\rangle$ is a Lie-Quillen model of $F^\xi_*(X,Y)$.
\end{enumerate}
\end{enumerate}
\end{theorem}
\begin{proof}
We have a natural bijection
\[
\pi_0F(X,Y)\cong [A(Y),A(X)]\cong \widetilde{\MCmod}(L(Y),A(X)_+).
\]
Similarly,
\[
\pi_0F_*(X,Y)\cong [A(Y),A(X)]_*\cong{\MCmod}(L(Y),A(X)_+)
\]
which proves (1).
Let us now prove (2) starting with part (a). Let $Z$ be any nilpotent rational CW complex $Z$ of finite type. We will prove the following natural bijection of sets:
\begin{equation}\label{adj}
[Z,F^\xi(X,Y)]_*\cong [L(Z),[{A(X)\otimes L(Y)]^\xi}\langle 0\rangle].
\end{equation}
Here the left hand side of (\ref{adj}) is the set of pointed homotopy classes of maps of spaces whereas the right hand side is  the set of homotopy classes of dgla maps. By the Yoneda lemma this would imply the desired statement. Note first that $[Z,F^\xi(X,Y)]_*$ can be identified with the fiber over $\xi\in[X,Y]$ of the map (induced by the inclusion of the base point into $Z$):
\[
[Z\times X,Y]\to [X,Y].
\]
By part (1) this is the same as the fiber over $\xi\in\widetilde{\MCmod}(L(Y),A(X))$ of the map
\[\widetilde{\MCmod}(L(Y), A(Z)\otimes A(X))\to \widetilde{\MCmod}(L(Y), A(X)).\]
We have:
\begin{align}
\nonumber[L(Z),[A(X)\otimes L(Y)]^\xi\langle 0\rangle]&\cong [\L A(Z)_+,[A(X)\otimes L(Y)]^\xi\langle 0\rangle]\\
\nonumber&\cong\MCmod([A(X)\otimes L(Y)]^\xi\langle 0\rangle, A(Z)_+)\\
&\label{fib0}\cong\MCmod(A(X)\otimes L(Y)^\xi, A(Z)_+)
\end{align}
where in the last equality we have used Proposition \ref{conn}.

Next consider the following sequence of maps of sets:
\begin{equation}\label{fib1}
\MCmod(A(X)\otimes L(Y)^\xi, A_+(Z))\to \widetilde{\MCmod}(A(X)\otimes L(Y)^\xi, A(Z))\to \widetilde{\MCmod}(L(Y), A(Z))
\end{equation}
where the first map is induced by the natural inclusion $A_+(Z)\to A(Z)$ and the second -- by the projection $A(Z)\to \mathbb Q$. This sequence is exact in the sense that second map is onto whereas the first term  is its fiber over $\xi\in\widetilde{\MCmod}(L(Y), A(X))$. Next we have natural bijections
\begin{equation}\label{fib2}
\widetilde{\MCmod}(A(X)\otimes L(Y)^\xi, A(Z))\cong \widetilde{\MCmod}(A(X)\otimes L(Y), A(Z))\cong \widetilde{\MCmod}(L(Y), A(Z)\otimes A(X)).
\end{equation}
Here the first bijection is given for $\eta\in A(X)\otimes L(Y)\otimes A(Z)$ by $\eta\mapsto \eta+\xi\otimes 1$ and the second one is obvious.

From equations (\ref{fib0}), (\ref{fib1}) and (\ref{fib2}) we conclude that the set $[L(Z),[A(X)\otimes L(Y)]^\xi\langle 0\rangle]$ is bijective with the fiber of \[\widetilde{\MCmod}(L(Y), A(Z)\otimes A(X))\to \widetilde{\MCmod}(L(Y), A(Z))\] over $\xi\in\widetilde{\MCmod}(L(Y), A(Z))$ and so the bijection (\ref{adj}) is proved.

The proof of the based version (b) follows from the unbased version (a). Note, first of all, that there is a homotopy fiber sequence of spaces
\[
F_*^\xi(X,Y)\to F^\xi(X,Y)\to Y
\]
where the last map is induced by the inclusion of the base point into $X$. Next, under the functor $?\mapsto L(?)$ the map $F^\xi(X,Y)\to $ corresponds to
\[
[A(X)\otimes L(Y)]^\xi\langle 0\rangle\to L(Y)
\]
induced by the augmentation $A(X)\to \mathbb Q$. Clearly the homotopy fiber of the latter map (which coincides in this case with the actual fiber) is isomorphic to $[A(X)\otimes L(Y)]^\xi\langle 0\rangle$.
\end{proof}
\begin{rem}
Note the manifest homotopy invariant nature of Theorem \ref{main}: the Quillen-Lie models of $F(X,Y)$ and $F_*(X,Y)$ do not depend (up to quasi-isomorphism) on the choice of models $L(Y)$ and $A(X)$ and also on the choice of the MC element $\xi$ within its homotopy class.
\end{rem}
Note that there is a weak equivalence between formal cdgas:
\[A(X)_+\cong \C(L(X)).\]
By Proposition \ref{hom2} it follows that there are $L_\infty$ quasi-isomorphisms
\[
(A(X)_+\otimes L(Y))^\xi\cong (\C(L(X))\otimes L(Y))^{\tilde{\xi}};
\]
here $\tilde{\xi}$ is the MC element in $(\C(L(X))\otimes L(Y))$
corresponding to a given map $X\to Y$. By definition
$(\C(L(X))\otimes L(Y))^{\tilde{\xi}}\cong \overline{C}_{\CE}(L(X),L(Y))$.
Similarly we obtain the following natural isomorphism of dglas:
\[
(A(X)\otimes L(Y))^\xi\cong {C}_{\CE}(L(X),L(Y)).
\]
This gives an interpretation of the Quillen-Lie algebras of $F(X,Y)$ and $F_*(X,Y)$ in terms of Chevalley-Eilenberg complexes. By Proposition
\ref{CEH} we can further interpret them in terms of Harrison complexes. Putting everything together we obtain the following result.
\begin{cor}\ 
\begin{enumerate}
\item
The dglas $C_{\Harr}(A(Y),A(X))\langle 0\rangle$ and $C_{\CE}(L(X),L(Y))\langle 0 rangle$ are Lie-Quillen models for a connected component of $F(X,Y)$.
\item
The dglas $\overline{C}_{\Harr}(A(Y),A(X))\langle 0\rangle$ and $\overline{C}_{\CE}(L(X),L(Y))\langle 0\rangle$ are Lie-Quillen models for a connected component of $F_*(X,Y)$.
\end{enumerate}
\end{cor}
\noproof
One further consequence is that the homotopy groups of function spaces can be expressed in terms of Harrison or Chevalley-Eilenberg cohomology; this was already known \cite{BL}:
\begin{cor} We have the following isomorphisms for $i>0$:
\begin{enumerate}
\item
\[
\pi_nF(X,Y)\cong H_{\CE}^{1-n}(L(X),L(Y))\cong H_{\Harr}^{1-n}(A(Y),A(X)).
\]
\item
\[
\pi_nF(X,Y)\cong \overline{H}_{\CE}^{1-n}(L(X),L(Y))\cong \overline{H}_{\Harr}^{1-n}(A(Y),A(X)).
\]
\end{enumerate}
\end{cor}
\noproof
Theorem \ref{main} can also be used to obtain Sullivan models of mapping spaces:
\begin{cor}\label{Hae}\
\begin{enumerate}
\item
The cdga $C_{\CE}([A(X)\otimes L(Y)]^\xi)\langle 0\rangle$ is a Sullivan model for $F(X,Y)$.
\item
The cdga $\overline{C}_{\CE}([A(X)\otimes L(Y)]^\xi\langle 0\rangle)$ is a Sullivan model for $F_*(X,Y)$.
\end{enumerate}
\end{cor}
\noproof
\begin{cor}
Let $X$ be a formal space (i.e. $A(X)$ is quasi-isomorphic to $H(X)$), $Y$ be a coformal space (i.e. $L(Y)$ is quasi-isomorphic to $\pi_{\mathbb Q}(Y)$, the Whitehead Lie algebra of $Y$) and $X\to Y$ be the constant map corresponding to $0\in\MC(L(Y),A(X))$.  Then the function spaces $F^0(X,Y)$ and $F_*^0(X,Y)$ are both coformal; their Whitehead Lie algebras are isomorphic respectively to $H^*(X)\otimes\pi_{\mathbb Q}(Y)$ and $H^*(X)_+\otimes\pi_{\mathbb Q}(Y)$.
\end{cor}
\noproof
\begin{rem}
Note that the Sullivan models described in the above corollary will not, in general, be minimal. However if $A(X)$ is formal (in the sense of being quasi-isomorphic to its cohomology) a minimal model for the components of $F^0(X,Y)$ and $F^0_*(X,Y)$ containing the constant map can be constructed. Indeed, consider instead of a dgla $L(Y)$ a minimal $L_\infty$ algebra $V$ equivalent to it (effectively, a Sullivan minimal model of $X$). Then $(H(X)\otimes V)\langle 0 \rangle$ is itself a minimal $L_\infty$ algebra whose representing formal cdga $\hat{S}\Sigma^{-1}[(H(X)\otimes V)\langle 0 \rangle]^*$ is a Sullivan minimal model of $F^0(X,Y)$; a similar argument applies to construct a Sullivan minimal model of $F^0_*(X,Y)$.\\ \\
\end{rem}
\begin{example}\
\begin{itemize}
\item
Consider the component of the constant map in $F(S^n,X)$ with $n\geq 1$. It follows from Theorem \ref{main} that a
Lie model of $F^0(X,Y)$ is the dgla $L(X)\ltimes\Sigma^{-n}L(X)\langle 0\rangle$,
the square-zero extension of $L(X)$ by $\Sigma^{-n}L(X)\langle 0\rangle$. It further follows from Corollary \ref{Hae}
that a Sullivan model of $F^0(S^n,X)$ is the cdga
\[C_{\CE}(L(X)\ltimes\Sigma^{-n}L(X)\langle 0\rangle)\cong C_{\CE}(L(X),\hat{S}\Sigma^{-1}[\Sigma^{-n} L(X)\langle 0\rangle]^*),\]
 the Chevalley-Eilenberg complex of $L(X)$ with coefficients in the (completed) symmetric algebra of
$\Sigma^{-1}[\Sigma^{-n} L(X)\langle 0\rangle]^*$. In the case $X$ is $n$-connected we have
$[\Sigma^{n}L(X)\langle 0\rangle]^*=[\Sigma^{n}L(X)]^*$ and so we get $C_{\CE}(L(X),\hat{S}[\Sigma^{n-1} L(X)^*])$ as a Sullivan model of $F(S^n,Y)$. Specializing further to the case $n=1$ we conclude that the cdga $C_{\CE}(L(X),\hat{S}[L(X)^*])$ is a Sullivan model of the free loop space of $X$. Note that the universal enveloping algebra of $L(X)$ is quasi-isomorphic to $C_*(\Omega(X))$, the chain algebra of the based loop space of $X$ whereas $\hat{S}[L(X)^*]$ is quasi-isomorphic to its dual $C_*(\Omega(X))$ and so we obtain a quasi-isomorphism
\[
C_{\CE}(L(X),S[L(X)^*])\cong C_{\Hoch}^*(C_*\Omega(X),C^*\Omega(X)).
\]
This is in agreement with the well-known result of Burghelea-Fiedorowicz and Goodwillie \cite{B,G}:
\[
H_*(F_*(S^1,X))\cong H_*^{\Hoch}(C_*\Omega(X),C_*\Omega(X)).
\]
\item
Now consider the component of the constant map in $\Omega^n(X)=F^0_*(S^n,X)$ with $n\geq 1$. Then a similar, but simpler, argument shows that the Lie-Quillen model of $\Omega^n(X)$ is the abelian Lie algebra $\Sigma^{-n}L(X)\langle 0\rangle$ and therefore its Sullivan model is the symmetric algebra
$S[\Sigma^{n-1}L(X)^*\langle 0\rangle])$. To be sure, this result is fairly obvious a priori.
\end{itemize}
\end{example}

\end{document}